\documentclass[11pt,oneside]{amsart}
\usepackage{amsmath,ifthen, amsfonts, amssymb, srcltx,
   amsopn, textcomp}

\usepackage{hyperref}
\usepackage{graphicx}

\usepackage[small]{caption}

\newcommand{\showcomments}{yes}
\renewcommand{\showcomments}{no}

\newcommand{\hidetodo}[1]
{\ifthenelse{\equal{\showcomments}{yes}}%
{#1}
}

\newsavebox{\commentbox}
\newenvironment{com}%
{\ifthenelse{\equal{\showcomments}{yes}}%
{\footnotemark
        \begin{lrbox}{\commentbox}
        \begin{minipage}[t]{1.25in}\raggedright\sffamily\tiny
        \footnotemark[\arabic{footnote}]}
{\begin{lrbox}{\commentbox}}}%
{\ifthenelse{\equal{\showcomments}{yes}}%
{\end{minipage}\end{lrbox}\marginpar{\usebox{\commentbox}}}
{\end{lrbox}}}

\newtheorem{thm}{Theorem}[section]
\newtheorem{lem}[thm]{Lemma}

\newtheorem{cor}[thm]{Corollary}

\newtheorem{prop}[thm]{Proposition}

\theoremstyle{definition}
\newtheorem{defn}[thm]{Definition}
\newtheorem{rem}[thm]{Remark}
\newtheorem{exmp}[thm]{Example}

\DeclareMathOperator{\rank}{rank}

\DeclareMathOperator{\Aut}{Aut}

\DeclareMathOperator{\stab}{Stab}

\DeclareMathOperator{\redrank}{\widetilde{rank}}

\newcommand{\homology}{\ensuremath{{\sf{H}}}}

\newcommand{\field}[1]{\mathbb{#1}}
\newcommand{\integers}{\ensuremath{\field{Z}}}

\newcommand{\reals}{\ensuremath{\field{R}}}

\newcommand{\boundary}   {{\ensuremath \partial}}

\newcommand{\euler}{\chi}

\newcommand{\p}{\textup{\textsf{p}}}

\newcommand{\edges}{\mathcal{E}}
\newcommand{\vertices}{\mathcal{V}}

\begin{document}

\title[Counting cycles in labeled graphs]{Counting cycles in labeled graphs: The nonpositive immersion property for one-relator groups}
\author{Joseph Helfer }
           \address{Dept. of Math. \& Stats.\\
                    McGill Univ. \\
                    Montreal, QC, Canada H3A 0B9 }
\email{joseph@helfer.ca \, \, \, wise@math.mcgill.ca}
\author[D.~T.~Wise]{Daniel T. Wise}
\date{\today}
\thanks{Research supported by NSERC}

\maketitle

\begin{com}
{\bf \normalsize COMMENTS\\}
ARE\\
SHOWING!\\
\end{com}

\begin{abstract}
  We prove a rank~1 version of the Hanna Neumann Theorem.  This shows
  that every one-relator $2$-complex without torsion has the nonpositive
  immersion property.  The proof generalizes to staggered and
  reducible $2$-complexes.
 \end{abstract}

\section{Introduction}

A \emph{deterministically labeled digraph} $\Gamma$ is a nonempty
graph whose edges are directed and labeled with the letters from an
alphabet $\{a_1, a_2, \ldots \}$, with the additional property that at
each vertex of $\Gamma$, no two outgoing edges have the same label,
and no two incoming edges have the same label. These are termed
 \emph{inverse automata} in the computer science literature. Let $w$ be
a nonempty word in $\{a_1^{\pm1},a_2^{\pm1}, \ldots\}$ that is
\emph{reduced} in the sense that no two consecutive letters of $w$ are
inverse to each other, and the first and last letters of $w$ are not
inverse to each other. We assume that $w$ is \emph{simple} in the
sense that $w\neq v^p$ for any word $v$ and $p>1$.  A \emph{$w$-cycle}
in $\Gamma$ is a closed based path in $\Gamma$ whose label is of the
form $w^n$ for some $n\geq 1$.  Two $w$-cycles in $\Gamma$ are
\emph{equivalent} if there is a path with label $w^m$ joining their
initial vertices for some $m\geq 1$.  The number of equivalence
classes of $w$-cycles in $\Gamma$ is denoted by $\overline
\#_w(\Gamma)$.  Finally, let $\beta_1(\Gamma) =
\rank(\homology_1(\Gamma))$ be the first Betti number of $\Gamma$, and
recall that $\beta_1(\Gamma) = |\edges(\Gamma)|-|\vertices(\Gamma)|+
1$ when $\Gamma$ is finite, connected, and nonempty. We use the
notation $\edges(Y)=\operatorname{Edges}(Y)$ and
$\vertices(Y)=\operatorname{Vertices}(Y)$ for a complex $Y$.

\begin{figure}[t]\centering
\includegraphics[width=.25\textwidth]{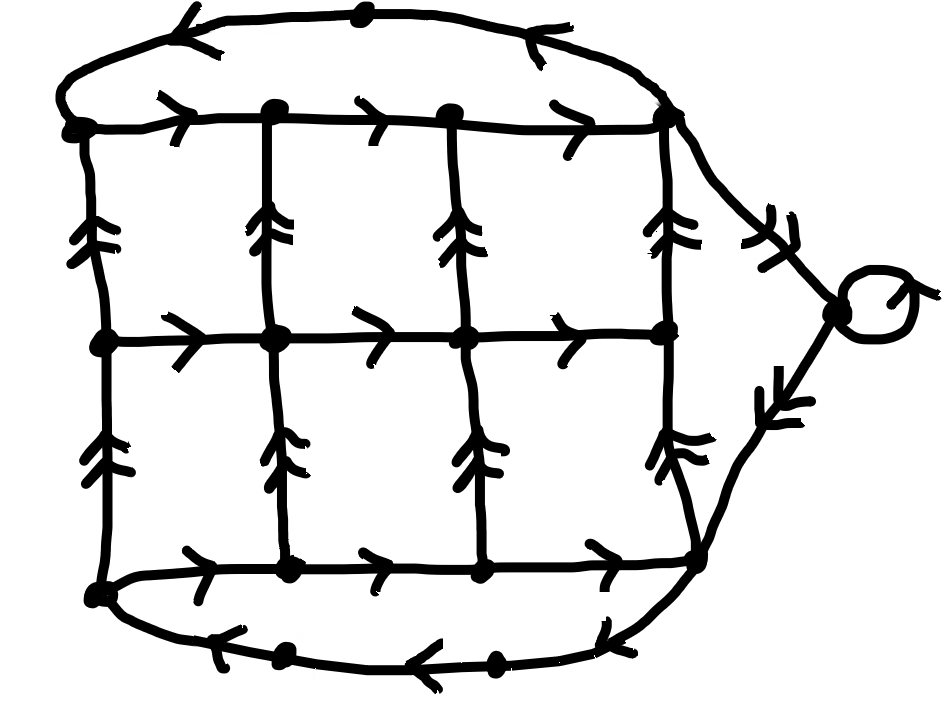}
\caption{\label{fig:NotACat}
Some values of $\#_w$ and $\overline \#_w$ for a deterministically labeled digraph.}
$\begin{array}{c||c|c|c|c}
               & w=aba^{-1}b^{-1} & w=b & w=a^{30}b^2a^{30}b^{-2} & w=a \\
 \hline
 \hline
\#_w           & 6 & 4 & 5& 12\\[.5ex]
\overline \#_w &  6 & 1 & 5& 3
\end{array}$
\vspace{-3mm}
\end{figure}

In this paper we prove the following naive statement illustrated in Figure~\ref{fig:NotACat}:
\begin{thm}\label{thm:word main}
  Let $\Gamma$ be a deterministically labeled finite digraph, and let
  $w$ be a reduced simple word in its alphabet.  Then $\overline
  \#_w(\Gamma) \leq \beta_1(\Gamma)$.
\end{thm}

A proof is given at the end of Section~\ref{sec:bislim}.

Theorem~\ref{thm:word main} was conjectured in
\cite{WiseAnnouncement03}, as part of a program to prove that every
one-relator group is coherent. See Remark~\ref{rem:coherence}.
Theorem~\ref{thm:word main} was proven when $w$ is a \emph{positive}
word in the sense that it has no $a_i^{-1}$ in
\cite{WisePositiveRelatorCoherence02}.  The inequality $\overline
\#_w(\Gamma) \leq 2\beta_1(\Gamma)$ was proven in
\cite{WiseHNCoherence} under the assumption that the Strengthened
Hanna Neumann Conjecture holds.  This latter conjecture was recently
proven in \cite{Friedman2015, Mineyev2012, DicksOnePageProof2011}.
The spirit of Dicks' proof which was extracted from Mineyev's
argument, and the realization that orderability should play a critical
role here, has inspired this note. The connection with the
Strengthened Hanna Neumann Theorem and a sense in which this is a
rank-one version of it is explained in Section~\ref{sec:shnc}, which
can be read independently of the other sections.

Lars Louder and Henry Wilton have independently proven Theorem~\ref{thm:word main} in \cite{LouderWilton2014}.
Their lovely proof is more geometrically palpable than ours,
yet also relies on orderability in a fundamental way.

\begin{defn}\label{defn:npi}
  A $2$-complex $X$ has \emph{nonpositive immersions} if for every
  combinatorial immersion $Y\rightarrow X$ with $Y$ compact and
  connected, either $\euler(Y)\le 0$ or $Y$ is contractible.
  We refer to \cite{WiseNonPositiveCoherence02} for a variety of classes of $2$-complexes with nonpositive immersions.
\end{defn}

The motivation for Theorem~\ref{thm:word main} is the following consequence which is a special case of
Theorem~\ref{thm:npi superduper}:

\begin{thm}\label{thm:npi}
Let $X$ be a $2$-complex with a single $2$-cell whose attaching map is not homotopic to a path of the form $v^n$ where $n>1$ and $v\rightarrow X^1$ is a closed path.
Then $X$ has nonpositive immersions.
\end{thm}

Theorem~\ref{thm:word main} is a simplified statement of results
simultaneously counting multiple types of immersed cycles in
Theorems~\ref{thm:bislim}~and~\ref{thm:main}.  These apply to the
class of ``bi-slim'' $2$-complexes which include staggered $2$-complexes,
and to the more general class of ``slim'' $2$-complexes (see
Definition~\ref{defn:slim}). In each case, an additional collapsing
conclusion is obtained, showing that these $2$-complexes have
nonpositive immersions.
In Section~\ref{sec:torsion}
we describe a way to treat the nonpositive immersion property for $2$-complexes whose $\pi_1$ is not torsion-free.
In Section~\ref{sec:reduced} we show that adding a $1$-cell and
$2$-cell to a slim $2$-complex usually results in another slim
$2$-complex. We deduce that Howie's reducible $2$-complexes are slim when
they have torsion-free $\pi_1$. Section~\ref{sec:shnc} contains an algebraic reformulation and consequence of Theorem~\ref{thm:word main}.
\section{Definitions}
\subsection{Preorder}
A \emph{preorder} on a set $E$ is a reflexive, transitive relation on
$E$, denoted by $\preceq$. As usual, $a\prec b$ means that
$(a\preceq b)\wedge\neg(b\preceq a)$. An element $a\in S$ is
\emph{minimal} in a subset $S\subseteq E$ if there is no $s\in S$ such
that $s\prec a$. The element $a$ is \emph{strictly maximal} in $S$ if
there is no $s\in S-\{a\}$ such that $a\preceq s$. The reader should
keep in mind the special case of a total ordering.

\subsection{Slim and bi-slim $2$-complexes}
\begin{defn}[Slim and bi-slim]\label{defn:slim}
  A combinatorial $2$-complex $X$ is \emph{slim} if:
  \begin{enumerate}
  \item\label{slim:order} There is a $\pi_1X$-invariant preorder on
    $\edges(\widetilde X)$.
  \item\label{slim:trav once} $\edges(\boundary \widetilde R)$ has a
    unique\begin{com}
      We needn't necessarily unique; distinguished may suffice. This
      may lead to interesting generalizations.
    \end{com} strictly maximal edge $e_{\widetilde R}^+$ for each $2$-cell $\widetilde R$ of
    $\widetilde X$.\\ Moreover, $e_{\widetilde R}^+$ is traversed exactly once by
    the boundary path $\boundary_\p R$.
  \item\label{slim:distinct} If $\widetilde R_1$ and $\widetilde R_2$ are distinct $2$-cells
    in $\widetilde X$ and $e_{\widetilde R_1}^+$ lies in $\boundary \widetilde R_2$ then
    $e_{\widetilde R_1}^+\prec e_{\widetilde R_2}^+$.
  \end{enumerate}
\end{defn}
$X$ is \emph{bi-slim} if it additionally satisfies:
\begin{enumerate}
  \setcounter{enumi}3
\item\label{bislim:bislim}
  for each $2$-cell $\widetilde R$ in $\widetilde X$, there is a
  distinguished edge $e_{\widetilde R}^-$
  such that for  distinct $2$-cells $\widetilde R_1$ and $\widetilde R_2$, if $\boundary_\p \widetilde R_1$ traverses
  $e_{\widetilde R_2}^-$, then $e_{\widetilde R_1}^+\prec e_{\widetilde R_2}^+$.
\end{enumerate}

\begin{exmp}
  Let $\varphi:J\to J$ be a $\pi_1$-injective map from a graph to
  itself. We show that the mapping torus $X$ of $\varphi$ is slim. The
  attaching map of each $2$-cell in $X$ is of the form
  $t_ua^{-1}t_v^{-1}\phi(a)$, where $a$ is a ``vertical'' edge arising
  from $J$ and each $t_p$ is a ``horizontal'' edge arising from a
  vertex $p$ of $J$. Let $\rho:\widetilde X\to\reals$ be the map
  associated to the homomorphism $\pi_1X\to\integers$ induced by
  $\pi_1J\mapsto0$ and $t_p\mapsto 1$. For edges $a,b$ of
  $\widetilde X$, we declare $a\preceq b$ if $\rho(a')\le\rho(b')$,
  where $a',b'$ are barycenters of $a,b$.
\end{exmp}

\begin{defn}\label{defn:staggered}
  A $2$-complex $X$ is \emph{staggered} 
   if its $2$-cells are totally ordered and a subset of its edges are totally ordered in such a way
  that
  \begin{enumerate}
  \item $\boundary_\p R$ is immersed and traverses at least one
    ordered edge for each $2$-cell $R$.
  \item If $R_1<R_2$ then $\min(R_1)<\min(R_2)$
    and $\max(R_1)<\max(R_2)$, where $\min(R)$ and $\max(R)$ denote
    the minimal and maximal edges in $\boundary R$.
      \end{enumerate}
  In particular, the standard $2$-complex $X$ of the presentation
  $\langle a,b,\ldots \mid W \rangle$ of a one-relator group without
  torsion is staggered and hence bi-slim by the following:
  \end{defn}


\begin{prop}\label{prop:staggered is bislim}
  Suppose $X$ is staggered and simple in the sense that no $2$-cell is
  attached along a proper power. Then $X$ is bi-slim.
\end{prop}
\begin{proof}
Since $\pi_1X$ is locally indicable \cite{Howie82}, it has a left-ordering $<$ by \cite{BurnsHale72}.

For each ordered $1$-cell $c$ of $X$, a
$\pi_1X$-invariant ordering of the cells of $\widetilde X$ mapping to
$c$ is induced by choosing a distinguished lift $\tilde c$, and declaring
$g_1\tilde c \prec g_2\tilde c$ when $g_1<g_2$.
If $g_1\widetilde c$ and $g_2\widetilde d$ are ordered cells that are not in the same orbit, then we declare
$g_1\widetilde c \prec g_2\widetilde d$ if their projections satisfy $c<d$.
We have thus produced a  $\pi_1X$-invariant total ordering on the edges of $\widetilde X$ that
 project to ordered edges of $X$, and so Property~\eqref{slim:order} holds,
 and $e_{\widetilde R}^+$ exists for each $\widetilde R$.
  To see that Property~\eqref{slim:trav once} holds, observe that
  $ \boundary_\p \widetilde R$ embeds in $\widetilde X$ for each $2$-cell $\widetilde R$. This is a well-known generalization of Weinbaum's Subword Theorem \cite{Weinbaum72} (c.f. Corollary~\ref{cor:tight}).
Property~\eqref{slim:distinct} obviously holds when $\widetilde R_1,\widetilde R_2$ project to different $2$-cells in $X$ since then $\max(R_1)<\max(R_2)$.
Property~\eqref{slim:distinct} holds when $\widetilde R_2=g\widetilde R_1$ with $g\neq 1$,
since $e^+_{\widetilde R_2}=ge^+_{\widetilde R_1} \neq e^+_{\widetilde R_2}$ and so $e^+_{\widetilde R_1}\prec e^+_{\widetilde R_2}$ by definition of $e^+_{\widetilde R_2}$.

Let $R$ denote a $2$-cell of $X$ with $c=\max(R)$.
Let $\tilde c$ denote the distinguished lift of $c$ declared above.
Let $\widetilde R$ be the lift of $R$ satisfying $e^+_{\widetilde R}=\tilde c$.
Let $\tilde b$ denote the distinguished lift of $b=\min(R)$.
 Let $\{h_k\tilde b\}$ be the translates of $\tilde b$ to edges of $\boundary \widetilde R$.
Consider the complete digraph $K$ with vertices $\{h_k\tilde b\}$ and with an edge directed from $h_i\tilde b$ to
$h_j\tilde b$ if $h_j^{-1} < h_i^{-1}$.
 The ordering of $\pi_1X$ ensures that $K$ is acyclic. We
declare $e^-_{\widetilde R}$ to be the source $h_s\tilde b$ of $K$.
For $g\in\pi_1X$ we declare $e_{g\widetilde R}^-=ge_{\widetilde R}^-$.

We now verify that Property~\eqref{bislim:bislim} holds.
If $\widetilde R_1,\widetilde R_2$ project to different cells in $X$,
then $\min(R_1)<\min(R_2)$ and so $R_1<R_2$ and so $\max(R_1)<\max(R_2)$ and hence $e^+_{\widetilde R_1}< e^+_{\widetilde R_2}$.
If $\widetilde R_1,\widetilde R_2$ project to the same cell $R$,
then $e^+_{\widetilde R_1}=g_1\tilde c$ and $e^+_{\widetilde R_2}=g_2\tilde c$ for some $g_1,g_2\in\pi_1X$ where $\tilde c$ is the distinguished lift of $c=\max(R)$.
Note that $e^-_{\widetilde R_2} = g_2e^-_{\widetilde R} = g_2h_s {\tilde b}$.
 Since $e^-_{\widetilde R_2}$ lies in $\boundary  \widetilde R_1$
we also have $e^-_{\widetilde R_2} = g_1 h_j{\tilde b}$ with $h_j^{-1}<h_s^{-1}$.
Comparing coefficients we have $(g_1h_j) = (g_2h_s)$.
We conclude that $e^+_{\widetilde R_1} < e^+_{\widetilde R_2}$
since $g_1 = (g_1h_j)h_j^{-1}<(g_2h_s)h_s^{-1} = g_2$.
\end{proof}

The referee observed that the ordering arising from the complete
digraph $K$ corresponds to the Duncan-Howie ordering for a one-relator
group \cite{DuncanHowie91}.

\subsection{$w$-cycles}
\begin{defn}[$w$-cycle]\label{defn:w-cycles}
  Let $X$ be a $2$-complex and $\{R_1,R_2,\ldots\}$ be its $2$-cells. For
  each $i$, let $w_i\to X^1$ be the immersed combinatorial circle
  corresponding to $\boundary_\p R_i$.
  Let $\Gamma\to X^1$ be an immersion of a connected nonempty graph. A
  \emph{$w_i$-cycle} in $\Gamma$ is a lift of $\widetilde w_i$ to
  $\Gamma$. Two such lifts are \emph{equivalent} if they differ by an
  element of $\Aut(\widetilde w_i)$.

  $\overline\#_{w_i}(\Gamma)$ is the number of equivalence classes of
  $w_i$-cycles in $\Gamma$ and
  $\overline\#_w(\Gamma)=\sum_i\overline\#_{w_i}(\Gamma)$.
  Likewise, $\#_{w_i}(\Gamma)$ is the number of $w_i$-cycles and
  $\#_w(\Gamma)=\sum_i\#_{w_i}(\Gamma)$.
  Note that the use of the symbols $\#_w$ and $\overline\#_w$ is
  interchanged with respect to their use in \cite{WiseHNCoherence}.
\end{defn}
\begin{com}
We could have dispensed with $\widetilde \Gamma$ and done
  everything in $(\widetilde X)^1$.
\end{com}
\subsection{Pre-widges}
We henceforth assume that $X$ is slim.
The preorder on the $\edges(\widetilde{X})$ induces a
$\pi_1X^1$-invariant preorder on the $\edges(\widetilde{X^1})$ via
the map $\widetilde{X^1}\to\widetilde X^1$.

Fixing basepoints of $X^1$ and $\Gamma$, we regard the universal cover
$\widetilde\Gamma$ as a subtree of $\widetilde{X^1}$, and we restrict
the above preorder to a $\pi_1\Gamma$-invariant preorder on the
$\edges({\widetilde\Gamma})$.

A \emph{$w_i$-line} is the image of a lift
$\widetilde w_i\hookrightarrow\widetilde\Gamma$ of a $w_i$-cycle in
$\Gamma$.  We use the term \emph{$w$-line} to indicate a $w_i$-line
for some $i$.

Let $\ell$ be a $w$-line. A \emph{pre-widge of $\ell$} is an edge that
is strictly maximal in $\edges(\ell)$. A \emph{pre-widge} is an edge
of $\widetilde{\Gamma}$ which is a pre-widge of some $w$-line.

As $\pi_1\Gamma$ permutes the $w_i$-lines (for each $i$) and preserves
the ordering, we see that $\pi_1\Gamma$ permutes the pre-widges. For
each $w$-line $\ell$, the pre-widges of $\ell$ lie in a single
$\stab_{\pi_1X^1}(\ell)$-orbit by
Definition~\ref{defn:slim}.\eqref{slim:trav once}. No edge is a
pre-widge of two different $w$-lines by
Definition~\ref{defn:slim}.\eqref{slim:distinct}.

\section{Widges and isles}

The image in $\Gamma$ of a pre-widge is a \emph{widge}.  Let
$\mathcal{W}(\Gamma)\subset \edges(\Gamma)$ denote the set of widges
in $\Gamma$. Removing the open edges $\mathcal{W}(\Gamma)$ from
$\Gamma$, we obtain a set $\mathcal{I}(\Gamma)$ of components called
\emph{isles}. The significance of the isle-widge decomposition lies in:

\begin{lem}\label{lem:no tree isle}
  Let $X$ be slim.  Let $\Gamma\to X^1$ be an immersion of a finite
  nonempty graph.  Then $\#_w(\Gamma)<\beta_1(\Gamma)+T$, where $T$ is
  the number of isles that are trees.
\end{lem}
\begin{proof}
  We first express the Euler characteristic of $\Gamma$ in terms of the
  decomposition:
  \[\chi(\Gamma)= \sum_{I\in \mathcal{I}(\Gamma)}\chi(I)\;-\;
  |\mathcal{W}(\Gamma)| \]
  hence
  \[|\mathcal{W}(\Gamma)|=-\chi(\Gamma)+\sum_{I\in
    \mathcal{I}(\Gamma)}\chi(I)= \beta_1(\Gamma) -1 +\sum_{I\in
    \mathcal{I}(\Gamma)}\chi(I)\]

  Since the last term is less than or equal to $T$, we have
  $|\mathcal{W}(\Gamma)|<\beta_1(\Gamma)+T$.

  Finally, $\#_w(\Gamma)=|\mathcal{W}(\Gamma)|$ by the definition of
  widge.
\end{proof}

\section{Counting $w$-cycles with multiplicity in the bi-slim
  case}\label{sec:bislim}
\begin{thm}\label{thm:bislim}
  Let $X$ be bi-slim.  Let $\Gamma\rightarrow X^1$ be an immersion of
  a finite nonempty graph.  Suppose each edge of $\Gamma$ is
  traversed by at least two $w$-cycles or traversed at least twice by
  some $w$-cycle. Then either $\Gamma$ is a single vertex or
  \begin{equation}\label{eq:main bi-slim}
    \#_w(\Gamma)<\beta_1(\Gamma).
  \end{equation}
\end{thm}
\begin{proof}
  By Lemma~\ref{lem:no tree isle}, the desired conclusion will follow by
  showing that no isle $I$ of $\Gamma$ is a tree. Suppose $I$ were a tree, and consider
  $\widetilde I \subset \widetilde{X^1}$.
  Choose $e$ to be a pre-widge of a $w$-line $\ell$
  intersecting $\widetilde I$, and assume that $e$ is minimal among all such choices.
  Observe that the edges $e_1,e_2$ in $\ell$ immediately before and
  after $\ell\cap \widetilde I$ are pre-widges of $\ell$.  Indeed,
  each $e_i$ is a pre-widge since it is incident to $\widetilde I$,
  and if $e_i$ were a pre-widge of another line, then $e_i\prec e$ by
  Definition~\ref{defn:slim}.\eqref{slim:distinct}, violating the
  minimality of $e$.
  Observe that neither $e_i$ maps to the distinguished edge $e_{\widetilde R}^-$ of
  the $2$-cell in $\widetilde X$ to whose boundary $\ell$ maps. Indeed,
  a second $w$-line $\ell'$ traversing $e_i$ would have a widge $q$
  satisfying $q\prec e$ by
  Definition~\ref{defn:slim}.\eqref{bislim:bislim}, contradicting the
  minimality of $\ell$.
  Thus the arc connecting $e_1,e_2$ contains an edge $e^-$ mapping to
  $e_{\widetilde R}^-$. Let $\ell'$ be another $w$-line that traverses $e^-$. Then
  $\ell'$ intersects $\widetilde I$ and $e' \prec e$ for any
  pre-widge $e'$ of $\ell'$ by
  Definition~\ref{defn:slim}.\eqref{bislim:bislim}. This contradicts
  the minimality of $e$.
  \end{proof}

\begin{defn}[Collapsing $\Gamma^w$]
  We form a $2$-complex $\Gamma^w$ from $\Gamma$ by adding a single
  $2$-cell for a representative of each $w$-cycle equivalence class in
  $\Gamma$.

  An edge $e$ in a $2$-complex is a \emph{free face} of a $2$-cell $f$ if
  $\boundary_\p f$ traverses $e$ exactly once and $e$ is not in the
  boundary of any other $2$-cell. In this case, we can \emph{collapse}
  to a subcomplex with the same homotopy type by removing the open
  cells $e,f$. If $B$ is obtained from $A$ by a sequence of such
  collapses then we say that $A$ \emph{collapses to} $B$.
\end{defn}

  \begin{cor}\label{cor:bi-slim collapse}
    Let $X$ be bi-slim.  Let $\Gamma\rightarrow X^1$ be an immersion
    of a finite nonempty graph. Then
    $\overline\#_w(\Gamma)\le\beta_1(\Gamma)$ with equality only if
    $\Gamma^w$ collapses to a tree.
  \end{cor}
  \begin{proof}
Suppose $\Gamma^w$ has an edge $e$ that is \emph{isolated} in the sense that $e$ is not in
the boundary of any $2$-cell. The statement holds for (each component
of) $\Gamma-e$ by induction on the number of such edges. When $e$ is
non-separating, the extra $\beta_1$ yields a strict inequality for
$\Gamma$. When $e$ is separating, either both components of
$\Gamma^w-e$ collapse to a tree, or we get a strict inequality for one
of the components of $\Gamma-e$ and hence for $\Gamma$.

We now assume that $\Gamma^w$ has no isolated edge. Suppose $\Gamma^w$
has a free face $e$. Observe that $e$ cannot be separating. Removing
$e$ decreases both $\overline\#_w$ and $\beta_1$ by 1. Hence, the
result holds for $\Gamma$ by induction on the number of edges. The
base-case of this induction holds by Theorem~\ref{thm:bislim},
since $\overline \#_w(\Gamma) \leq \#_w(\Gamma)$.
\end{proof}

\begin{proof}[Proof of Theorem~\ref{thm:word main}]
  A deterministically labeled digraph $\Gamma$, as described in the
  introduction, is equivalent to a combinatorial immersion of a graph
  into a bouquet of circles $X^1$. Attaching a $2$-cell to $X^1$ along a
  path corresponding to the cyclically reduced simple word $w$
  produces a bi-slim complex by Proposition~\ref{prop:staggered is bislim}.
  There is a bijection between $w$-cycles in the sense of the
  introduction and $w$-cycles in the sense of
  Definition~\ref{defn:w-cycles}, and this bijection respects the
  equivalence relations.
  Theorem~\ref{thm:word main} then follows immediately from
  Corollary~\ref{cor:bi-slim collapse}.
  Moreover, under the additional assumption that each edge of $\Gamma$
  lies in at least two $w$-cycles, Theorem~\ref{thm:bislim} gives
  the stronger inequality $\#_w(\Gamma)<\beta_1(\Gamma)$.
\end{proof}

\begin{cor}\label{cor:bi-slim nonpositive immersions}
  If $X$ is bi-slim, then $X$ has nonpositive immersions.
\end{cor}
\begin{proof}
  Let $Y\rightarrow X$ be an immersion with $\euler(Y)>0$.  Let $\Gamma=Y^1$. Let $F$
  be the number of $2$-cells in $Y$.
The first inequality in~\eqref{eq:three in a row} holds by Corollary~\ref{cor:bi-slim collapse},
the second holds since $\euler(Y)=1-\beta_1(\Gamma)+F$,
and the third holds by definition of $F$ and $ \#_w$.
\begin{equation}\label{eq:three in a row}
   \#_w(\Gamma)\le  \beta_1(\Gamma)\le F\le  \#_w(\Gamma)
  \end{equation}
Thus each inequality in~\eqref{eq:three in a row} is an equality, and so $Y=\Gamma^w$ collapses by the second statement of Corollary~\ref{cor:bi-slim collapse}.
\end{proof}
\begin{rem}[Coherence of one-relator groups]\label{rem:coherence}
  The original motivation for counting $w$-cycles, and in particular
  for proving Corollary~\ref{cor:bi-slim nonpositive immersions}, is to
  affirmatively answer G.~Baumslag's question on the coherence of
  one-relator groups. There is currently a gap in the proof of the main goal of \cite{WiseNonPositiveCoherence02} which asserts: if $X$ has nonpositive immersions then $\pi_1X$ is \emph{coherent} in the sense
   that every finitely generated subgroup of $\pi_1X$ is finitely presented.
\end{rem}

\section{Counting $w$-cycles (without multiplicity) in the slim case}
%
As we now only count equivalence classes of $w$-cycles, we
focus on only one widge from each equivalence
class of $w$-cycles.
Accordingly,  two widges in $\Gamma$ are \emph{equivalent} if they are
images of pre-widges of a common $w$-line. We arbitrarily
select one widge from each equivalence class and call these
\emph{great widges} and refer to their preimages as \emph{great pre-widges}.
The \emph{great isles} are the components obtained by removing the great widges from $\Gamma$.

\begin{lem}\label{lem:no tree great isle}
  Let $X$ be slim. Let $\Gamma\to X^1$ be an immersion of a finite
  nonempty graph. Then $\overline\#_w(\Gamma)<\beta_1(\Gamma)+
  \overline T$, where $\overline T$ is the number of great isles that
  are trees.
\end{lem}
\begin{proof}
  This is proved like Lemma~\ref{lem:no tree isle} replacing
  $\#_w(\Gamma)$ by $\overline\#_w(\Gamma)$ and  $T$ by $\overline T$.
  \end{proof}

A great widge is a \emph{local widge} to a great isle if its
corresponding $w$-cycle does not traverse an edge in any other
great isle.

\begin{lem}\label{lem:collapsing J is good}
  Let $\Gamma\to X$ be an immersion of a locally finite, connected
  graph. Let $I$ be a great isle that is a finite tree. Let $J$ be the
  union of $I$ and its local widges. If $\pi_1J\to\pi_1X$ has trivial
  image, then $I$ is the unique great isle.
\end{lem}\begin{com}this works more generally with:  If $\pi_1J\to\pi_1X$ has
  trivially ordered  image!\end{com}
\begin{proof}
  Consider $\widetilde J\subset \widetilde \Gamma$.  Since
  $\stab(\widetilde J)$ is a subgroup of
  $\ker(\pi_1X^1\rightarrow \pi_1X)$ we see that there are finitely
  many $\preceq$-equivalence classes of edges incident with $\widetilde J$ in $\widetilde \Gamma$.

  Let $\tilde e_1$ be a minimal great pre-widge with a single vertex in $\widetilde J$;
  its image $e_1$ is not a local widge. Consider its $w$-line
  $\ell_1$. Since $e_1$ is not a local widge of $I$, the line $\ell_1$
  contains another great pre-widge $\tilde e_2$ with a single vertex in $\widetilde J$.
  Since $\tilde e_2$ is on $\ell_1$, we must have
  $\tilde e_2\preceq \tilde e_1$. The minimality of $\tilde e_1$ obviates
  $\tilde e_2\prec \tilde e_1$.

  If $\tilde e_1\preceq \tilde e_2$ then $\tilde e_2$ is a pre-widge of $\ell_1$ since
  $\tilde e_1$ is a pre-widge. As $\tilde e_2$ is a great pre-widge, it must
  be a great pre-widge of $\ell_1$ and hence in the same $\stab(\ell_1)$-orbit as
  $\tilde e_1$. But then $\tilde e_1$ and $\tilde e_2$ have the same image in $\Gamma$,
  which is thus a local widge as the path from $\tilde e_1$ to $\tilde e_2$ is in
  $\widetilde J$.  This contradicts that $\tilde e_1$ is not a local widge.

As each great widge incident to $I$ is local, we see that $I$ is the only isle.
  \end{proof}

\begin{lem}\label{lem:unique tree isle}
  Let $\Gamma\to X$ be an immersion of a locally finite, connected
  graph. If some great isle is a finite tree, then it is the only
  great isle.
\end{lem}
\begin{proof}
  Let $\Gamma$ be a counterexample that is minimal in the sense that
  it has a tree great isle $I$ with fewest local widges
  among all tree great isles of all counterexamples.

  Let $J$ be the union of $I$ with all its local widges. Observe that
  $\pi_1J\to\pi_1X$ has nontrivial image, for otherwise,
  Lemma~\ref{lem:collapsing J is good} implies that $I$ is the only
  great isle.


  Let $\widehat \Gamma\to \Gamma$ be the cover corresponding to
  $\ker(\pi_1\Gamma\to\pi_1X)$.
 Each $w$-cycle of $\Gamma$ lifts to a
  $w$-cycle of $\widehat \Gamma$, and moreover, every $w$-cycle of $\widehat \Gamma$ arises in this way.
   We can thus choose the great widges of
  $\widehat \Gamma$ to be the pre-images of great widges of $\Gamma$.  Hence
  any pre-image $\widehat I$ of $I$ is a great isle of $\widehat \Gamma$.
  Since $\widehat I$ has fewer local widges than $I$ we obtain a
  smaller counterexample $\widehat \Gamma$.
\end{proof}

\begin{lem}\label{lem: unique isle collapses}
  Let $\Gamma\to X$ be an immersion of a finite connected graph. If
  $\Gamma$ has a great isle that is a finite tree, then $\Gamma^w$
  collapses to a tree.
\end{lem}
\begin{proof}
  We first verify that $\pi_1\Gamma\rightarrow\pi_1 X$ is trivial.
  Indeed, otherwise, the corresponding cover $\widehat \Gamma$ has
  $\deg(\widehat \Gamma\rightarrow \Gamma)$ finite tree great isles,
  violating Lemma~\ref{lem:unique tree isle}.

  Since every widge in $\Gamma$ is local to its unique tree great
  isle, we conclude, as in the proof of Lemma~\ref{lem:collapsing J is
    good}, that there are finitely many $\preceq$-equivalence classes
  of edges in $\widetilde\Gamma$. Let $\tilde e$ be a maximal
  pre-widge in $\widetilde\Gamma$. Then by maximality, $\tilde e$
  cannot lie in another $w$-line. Hence its image $e$ in $\Gamma$ is a
  free face in $\Gamma^w$ and the $2$-cell on which $e$ lies can be
  collapsed. Repeating this process, we collapse $\Gamma^w$ to a tree.
\end{proof}

%
%
\begin{thm}\label{thm:main}
  Let $X$ be slim. Let $\Gamma\to X^1$ be an
  \begin{com}If disconnected then it collapses to a forest\end{com}
  immersed finite connected nonempty graph. Then:
  \begin{equation}\label{eq:main}
    \overline\#_w(\Gamma)\le \beta_1(\Gamma)
  \end{equation}
  Moreover, if $\overline\#_w(\Gamma)=\beta_1(\Gamma)$, then
  $\Gamma^w$ collapses to a tree.
\end{thm}
\begin{proof}
  By Lemma~\ref{lem:unique tree isle}, at most one great isle of
  $\Gamma$ is a tree. Hence $\overline \#_w(\Gamma)\le\beta_1(\Gamma)$ by
  Lemma~\ref{lem:no tree great isle}.
  The ``moreover'' part follows from Lemma~\ref{lem: unique isle collapses}.
\end{proof}

\begin{cor}\label{cor:slim nonpositive immersions}
  If $X$ is slim, then $X$ has nonpositive immersions.
\end{cor}
\begin{proof}
Replace Corollary~\ref{cor:bi-slim collapse} by Theorem~\ref{thm:main}
in the proof of Corollary~\ref{cor:bi-slim nonpositive immersions}.\end{proof}

\section{Nonpositive immersions when there is torsion}\label{sec:torsion}
There are various ways of obtaining a version of nonpositive immersions for a one-relator group with torsion.
Since one-relator groups with torsion are virtually torsion free (this holds more generally for staggered presentations where all relators are proper powers) the following provides a useful interpretation:
\begin{thm}\label{thm:npi superduper}
 Let $X_*$ be a $2$-complex whose $i$-th $2$-cell has attaching map $w_i^{n_i}$ where $w_i$ is not a proper power.
Let $X$ be the $2$-complex with $X^1 = X_*^1$ and whose $i$-th $2$-cell has attaching map $w_i$.
Let $\widehat X_*\rightarrow X_*$ be a finite regular cover such that no $2$-cell is attached along a proper power.
Let $Z$ be a subcomplex of $\widehat X_*$ that contains exactly one
$2$-cell from each set of $n_i$ $2$-cells attached along a lift of
$w_i^{n_i}$.

If $X$ is slim then $Z$ has nonpositive immersions.

If $X$ is bi-slim, and let $n=\min_i(n_i)$.
Then for any collapsed immersion $Y\rightarrow Z$, either $Y$ has an isolated edge, or $Y$ is a vertex,
 or $\euler(Y)\leq -(n-1)|\text{$2$-cells}(Y)|$.
 \end{thm}
\begin{proof}
  Let $Y\rightarrow Z$ be an immersion with $Y$ compact and connected.
  Then the induced map $Y^1\rightarrow X$ is also an immersion, and
  each $w$-cycle in $Y^1$ with respect to $Z$ is also a $w$-cycle with
  respect to $X$. If $X$ is slim, then Theorem~\ref{thm:main} holds
  for $X$, and it follows that the conclusion of
  Theorem~\ref{thm:main} holds for $Z$ as well, and hence that $Z$ has
  nonpositive immersions (by the proof of Corlloary~\ref{cor:slim
    nonpositive immersions}).
\begin{com}
Direct proof (instead of reusing other proof):

We may suppose that $Y$ has no isolated edges, and since collapsing along free faces does not affect the euler characteristic, we may suppose that $Y\rightarrow Z$ is collapsed. Thus each edge of $Y^1$ is traversed in at least two ways by attaching maps of $2$-cells of $Y$.

Consider the map $Y^1\rightarrow X$. In the slim case, by Theorem~\ref{thm:main} either $Y^1$ is a vertex,
or $\overline\#_w(Y^1) < \beta_1(Y^1)$. Thus $\euler(Y) = \euler(Y^1)+|\text{$2$-cells}(Y)|
 \leq 1-\beta_1(Y^1)+\overline\#_w(Y^1) \leq 0$.
\end{com}

In the bi-slim case we note that, not only are the $w$-cycles of $Y^1$
with respect to $Z$ also $w$-cycles with respect to $X$, but each
$w_i$-cycle in $Y$ appears with multiplicity at least $n_i$. Since $X$
is bi-slim, if $Y$ is not a single vertex and has no isolated edge,
then by Theorem~\ref{thm:bislim} we have:
\[
n|\text{$2$-cells}(Y)|<\beta_1(Y^1)
\]
and hence
\[
\euler(Y)=(1-\beta_1(Y^1))+|\text{$2$-cells}(Y)|\le
(-n+1)|\text{$2$-cells}(Y)|.
\qedhere\]
\end{proof}

\section{Reducible is Slim}\label{sec:reduced}
\begin{defn}[Enlargement]
The connected combinatorial $2$-complex $Y$ is an \emph{$(R,e)$-enlargement} of the subcomplex $X$
if  $Y-X=R\cup e$ where $e$ is an open edge, and $R$ is an open $2$-cell,
and $\boundary_\p R$ traverses $e$ but $\boundary_\p R$ is not homotopic in $X\cup e$ to a path traversing $e$ fewer times.
Similarly, $Y$ is an \emph{$e$-enlargement} if $Y-X$ consists of a single open edge $e$.
An enlargement is \emph{simple} if either it is an $e$-enlargement, or it is an $(R,e)$-enlargement
and $\boundary_\p R$ is not homotopic to a proper power in $X\cup e$.
\end{defn}

Howie provided the following generalization of Weinbaum's subword
theorem \cite[Cor~3.4]{Howie82}:
\begin{lem}\label{lem:Howie's Weinbaum generalization}
 Let $Y$ be an $(R,e)$-enlargement of $X$.
 Suppose  $\boundary_\p R=P_1P_2$ where each $P_i$ is a closed path in $Y$
  that traverses $e$.
  Then each $P_i$ is essential in $Y$.
\end{lem}

\begin{lem}\label{lem:redinject}
Let $Y$ be an $(R,e)$-enlargement of $X$. Suppose $\pi_1Y$ is left-orderable. Then $X\to Y$
  is $\pi_1$-injective on each component of $X$.
\end{lem}
\begin{proof}
Arguing by contradiction, consider a minimal area disk diagram $D\to Y$ where
  $\boundary_\p D$ is an essential closed path in $X$.
  Choose a lift $D\to\widetilde Y$ to the universal cover of $Y$.
    We order the edges in $\widetilde Y$ mapping to $e$ by setting
  $g_1\tilde e<g_2\tilde e$ if $g_1<g_2$.
  Let  $f$ be a maximal edge in $\text{image}(D\to\widetilde Y)$  among all edges mapping to $e$.
  Let $f'$ be an edge of $D$ mapping to $f$.  Observe that
  $f' \subset \text{interior}(D)$ since $f'\not\subset \boundary D$ as
  $e\not\subset X$.  Let $S',T'$ denote the two $2$-cells of $D$ on
  opposite sides of $f'$, and let $S,T$ denote their images in
  $\widetilde Y$.  We will show below that $S\neq T$.  Let
  $g\in\pi_1 Y$ be the nontrivial element such that $g S=T$. Both $gf$
  and $g^{-1}f$ lie in $\text{image}(D\rightarrow \widetilde Y)$. This
  contradicts the maximality of $f$ since either $gf>f$ or
  $g^{-1}f>f$.

  We now reach a contradiction if $S= T$.  Consider the edges mapping
  to $f$ in a lift $\boundary_\p R\rightarrow \widetilde Y$.  If there
  is only one such edge, then $S',T'$ form a cancelable pair, and so
  the minimality of $D$ is violated.  If two of these edges are
  oriented in the same way around $\boundary_\p R$, then the subpath
  joining their initial vertices violates
  Lemma~\ref{lem:Howie's Weinbaum generalization}.  Thus exactly two edges of $\boundary_\p
  R$ map to $f$ so $\boundary_\p R\rightarrow \widetilde Y$ is of the
  form $\sigma f \sigma' f^{-1}$.  Moreover, projecting to $Y$, we
  find that $\sigma$ cannot traverse $e$, for it would violate
  Lemma~\ref{lem:Howie's Weinbaum generalization} and likewise for
  $\sigma'$. We conclude that $\boundary_\p R$ is of the form $\sigma
  e \sigma' e^{-1}$ where $\sigma$ and $\sigma'$ are closed essential
  paths in $X$.
  However, in this case, $\pi_1Y$ splits as an HNN extension or
  amalgamated product (along $\langle \sigma \rangle$), depending on
  whether or not $X$ is connected, and hence $X\to Y$ is
  $\pi_1$-injective on each component.
\end{proof}

\begin{thm}\label{thm:enlargement slim}
Let $Y$ be a simple enlargement of $X$.
Then $Y$ is slim if each component of $X$ is slim.
\end{thm}
\begin{proof}
  As each component of $X$ is slim, it has nonpositive immersions by
  Corollary~\ref{cor:slim nonpositive immersions}. And hence has
  locally indicable $\pi_1$ by \cite{WiseNonPositiveCoherence02}.
  As the enlargement is \emph{simple}, $\pi_1Y$ is locally
  indicable \cite[Cor~4.2]{Howie82} and thus has a left-ordering $<$.
  For an $(R,e)$-enlargement, any proper nontrivial subpath of $\boundary_\p R$ which starts and
  ends at an initial vertex of $e$ is essential in $\pi_1Y$ by
  Lemma~\ref{lem:Howie's Weinbaum generalization} and our assumption
  that $\boundary_\p R$ cannot be homotoped to traverse $e$ fewer
  times.  Furthermore, as $X\rightarrow Y$ is $\pi_1$-injective on
  each component by Lemma~\ref{lem:redinject}, we see that any subpath
  of $\boundary_\p R$ that starts and ends at an initial vertex of $e$
  is essential in $\pi_1Y$.

  To see that $Y$ is slim, we declare a $\pi_1Y$-invariant preorder on
  the edges of $\widetilde Y$ as follows: $\tilde c_1\preceq\tilde
  c_2$ if either:
  \begin{enumerate}
  \item $\tilde c_1$ maps to $X$ and $\tilde c_2$ maps to $e$.
  \item $\tilde c_1=g\tilde c_2$ with $g<1_{\pi_1Y}$ and both $\tilde
    c_1$ and $\tilde c_2$ map to $e$.
  \item $\tilde c_1,\tilde c_2$ lie in the same component of the
    pre-image of $X$ and $\tilde c_1\preceq\tilde c_2$ with respect to
    the slim structure on $X$.\qedhere
  \end{enumerate}
\end{proof}

\begin{defn}[Reducible]
  A $2$-complex $X=\cup_{i=0}^m X_i$ is \emph{$[$simply$]$ reducible} if
  $X_0$ is a vertex, and $X_{i+1}$ is a $[$simple$]$ enlargement of
  $X_i$ for each $i\geq 0$.  We allow $m=\infty$.
\end{defn}
Howie's original definition of reducible $2$-complex is a bit more general, as he imposes the laxer
  requirement that for an $(R,e)$-enlargement,  $\boundary_\p R$ is not homotopic
in $X\cup e$ to a path not traversing $e$. However, any $2$-complex satisfying Howie's definition
has the homotopy type of a reduced $2$-complex in the above sense.

\begin{cor}\label{cor:reduced slim}
Every simply reducible $2$-complex is slim.
\end{cor}
\begin{proof}This follows by induction  from Theorem~\ref{thm:enlargement slim}.
Note that the preorder on the edges in copies of $\widetilde X_i$ are in agreement with the preorder
 on $\widetilde X_{i+1}$ for each $i$, and hence the case $m=\infty$ holds as well.
 \end{proof}

An $(R,e)$-enlargement $X_{i+1}$ of $X_i$ is \emph{tight} if each subpath of $\boundary_\p R$ mapping to $X_i$ lifts to an embedding in $\widetilde X_i$.
And $X$ is a \emph{tight} reducible complex if each $(R,e)$-enlargement $X_i\subset X_{i+1}$ is tight.
The following is then a consequence of Lemma~\ref{lem:redinject}:
 \begin{cor}\label{cor:tight}
If $X$ is a tight reducible complex, then the boundary path of each $2$-cell embeds in $\widetilde X$.
\end{cor}

\section{Connection to the Strengthened Hanna Neumann Theorem}\label{sec:shnc}
In this section, we describe the connection of
Theorem~\ref{thm:word main} to the Strengthened Hanna Neumann Theorem.

\begin{defn}[Fiber product]
  Let $\Gamma_1,\Gamma_2$ be labeled digraphs. Their \emph{fiber product}
  is the labeled digraph whose vertices are pairs $(v_1,v_2)$ of vertices in $\Gamma_1,\Gamma_2$,
  and whose edges are pairs $(e_1,e_2)$ of edges in $\Gamma_1,\Gamma_2$ with the same label.
  The initial and terminal vertex of $(e_1,e_2)$ are $(u_1,u_2)$ and $(v_1,v_2)$ where
  $u_i,v_i$ are the initial and terminal vertices of $e_i$ in $\Gamma_i$.
\end{defn}

Fiber products were popularized in combinatorial group theory by Stallings
\cite{Stallings83}. The same construction, phrased in the
language of finite state automata, was widely used in computer science
to compute the intersection of regular languages.

\begin{defn}
  The \emph{reduced rank} of a graph $K$ is
  $\redrank(K) =  \ \max\{\beta_1(K)-1,\ 0\}$.
\end{defn}

The Strengthened Hanna Neumann Theorem is equivalent to the following
inequality, which was first stated explicitly by Walter Neumann in
\cite{WNeumann89}. It is this statement that was proven in
\cite{Friedman2015,Mineyev2012}.

\begin{thm}\label{thm:shnc}
  Let $\Gamma_1,\Gamma_2$ be connected deterministically
  labeled digraphs. Then:
\begin{equation}\label{eq:fiber product SHN}
  \sum_{K\in Components(\Gamma_1\otimes\Gamma_2)}\hspace{-1.3cm}
  \redrank(K)\ \ \le \ \ \redrank(\Gamma_1) \cdot \redrank(\Gamma_2)
  \end{equation}
\end{thm}
Letting $\Gamma_1$ be the cycle labeled by a cyclically reduced word
$w$ which is not a proper power, and letting $\Gamma_2$ be arbitrary,
Theorem~\ref{thm:word main} can be restated as
\[
\sum_{\substack{K\in Components(\Gamma_1\otimes\Gamma_2)}}\hspace{-1.3cm}
\beta_1(K)\ \ \le \ \ \beta_1(\Gamma_1)\cdot \beta_1(\Gamma_2)
\]
\begin{figure}
  \centering
  \includegraphics[width=.35\textwidth]{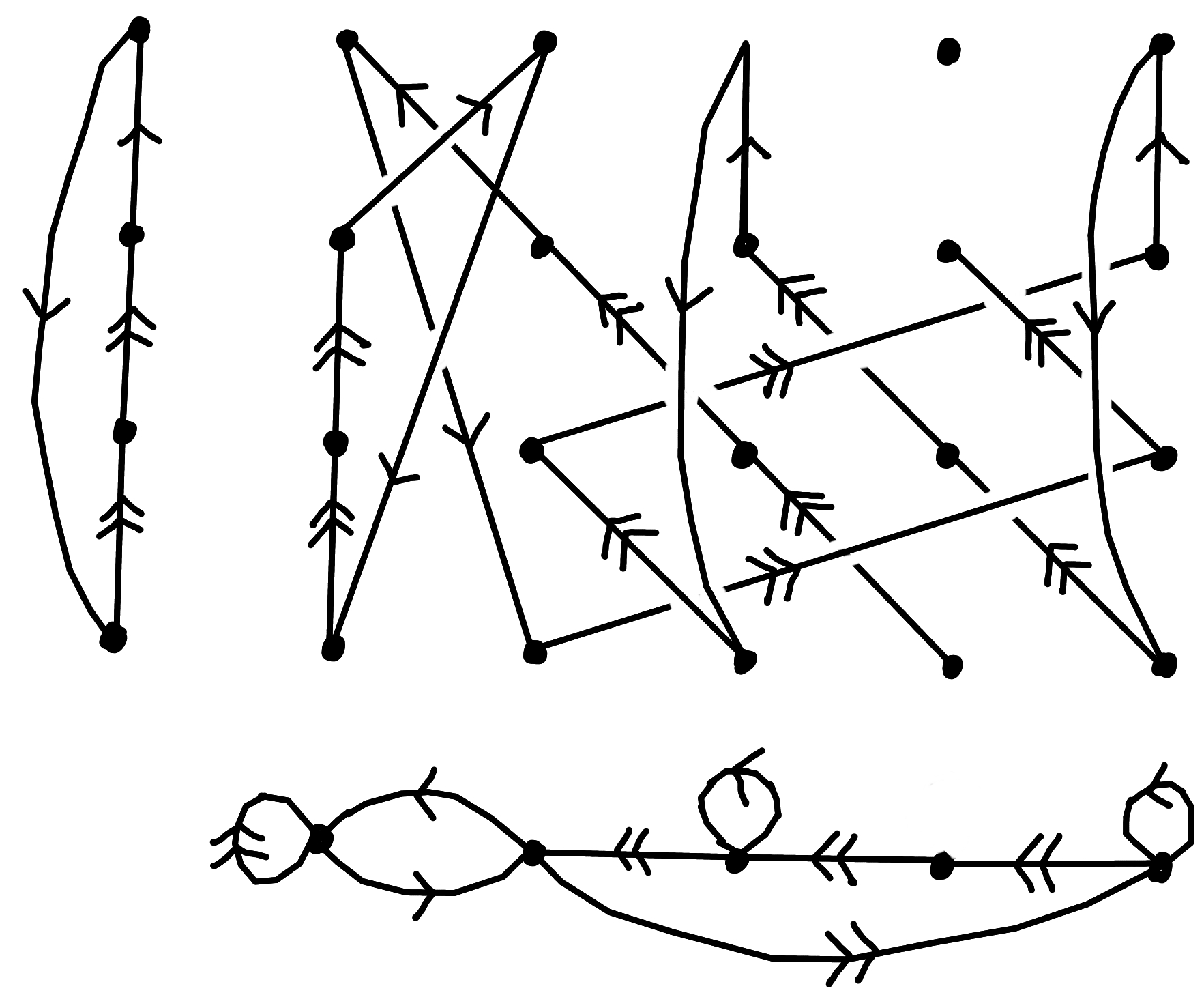}
  \caption{\label{fig:fiberproduct} The above fiber product of a graph and a circle contains an arc, an isolated vertex, and two cycles with multiplicities one and two.}
\end{figure}

Unfortunately, Theorem~\ref{thm:shnc} sheds no light on
Theorem~\ref{thm:word main}, since $\redrank(\Gamma_1)=0$ when $\Gamma_1$ is a cycle,
 and as illustrated in Figure~\ref{fig:fiberproduct},
 each component $K$ of $\Gamma_1\otimes\Gamma_2$ has $\redrank(K)=0$,
 so Equation~\eqref{eq:fiber product SHN} becomes $0\le0$.
Nevertheless, this connection was exploited in~\cite{WiseHNCoherence} to obtain a
partial result by choosing $\Gamma_1$ appropriately related to $w$ so that no component $K$ has $\euler(K)=0$.

In parallel to the statement of the (Strengthened) Hanna Neumann Theorem which was originally formulated in terms of intersections of subgroups of a free group, we have the following:
\begin{cor}[Restatement of Theorem~\ref{thm:word main}]\label{cor:algebraic version}
  Let $H$ be a finitely generated subgroup of a free group $F$. Let
  $Z\subset F$ be a maximal cyclic subgroup. Then the number of
  distinct conjugates of $Z$ that intersect $H$ nontrivially is
  bounded by $\rank(H)$.
\end{cor}
\begin{proof}
Let $F=\pi_1B$ where $B$ is a bouquet of circles. Let $\widehat B\rightarrow B$ be the based covering space with $\pi_1\widehat B=H$. Let $\Gamma\subset \widehat B$ be a finite connected based subgraph.
 Direct and label the edges of $B$, and pull this back so $\Gamma$ is a finite deterministically labeled digraph.
 Let $Z=\langle w\rangle$, where we may assume without loss of generality that $w$ is cyclically reduced.
Each conjugate of $Z$ that intersects $H$ nontrivially corresponds to a closed lift of some power $w^n$ of the path $w\rightarrow B$ at some vertex of $\Gamma$, and hence to a based $w$-cycle. Two based $w$-cycles in the same equivalence class correspond to vertices connected by a lift of a path $w^k$, and hence to the same conjugate of $Z$. The bound holds by Theorem~\ref{thm:word main}.
\end{proof}

\begin{cor}
  Let $H$ be a finitely generated subgroup of a free group $F$.
  Suppose $H$   is \emph{isolated} in the sense that $h^p\in F$ implies $h\in F$
  for any $h\in F$ and $p>0$. Suppose
  $Hg_1,\ldots,Hg_n$ are distinct cosets with $n> \rank(H)$. Then:
    \[
  \bigcap_{i=1}^ng_i^{-1}Hg_i = \{1_F\}
  \]
\end{cor}
\begin{proof}
Consider a nontrivial element $w\in \bigcap_{i=1}^ng_i^{-1}Hg_i$.
As $H$ is isolated, we may assume that $w$ is not a proper power.
Let $Z=\langle w\rangle$, and apply Corollary~\ref{cor:algebraic version} to
$\{g_iZg_i^{-1} : 1\leq i\leq n \}$.
Note that the maximal cyclic subgroup $Z$ is malnormal and so the conjugates are distinct.
  \end{proof}

{\bf Acknowledgment:} We are grateful to the referees for
many helpful corrections.

\bibliographystyle{alpha}

\bibliography{wise}
%
%

\end{document}